\documentclass[11pt]{amsart}
\usepackage{amssymb}
\usepackage{amsfonts}
\usepackage{etoolbox}
\usepackage{setspace}
\usepackage[hidelinks]{hyperref}
\usepackage[initials, shortalphabetic]{amsrefs}
\usepackage[left=3.2cm,right=3.2cm,bottom=3cm,top=3cm]{geometry}

\makeatletter\pretocmd{\@seccntformat}{\S}{}{}\makeatother
\newtheorem{theorem}{Theorem}[section]

\newtheorem{corollary}[theorem]{Corollary}

\newtheorem{lemma}[theorem]{Lemma}
\newtheorem*{lemma*}{Lemma}

\newtheorem*{claim*}{Claim}

\newtheorem{proposition}[theorem]{Proposition}

\theoremstyle{definition}

\newtheorem*{definition-no}{Definition}

\DeclareMathOperator{\Mod}{Mod}
\DeclareMathOperator{\Iso}{Iso}

\begin{document}
\title[L\'opez-Escobar and the topological Vaught conjecture]{A
L\'opez-Escobar theorem for metric structures, and the topological Vaught
conjecture}
\author{Samuel Coskey}
\address{Samuel Coskey\\
Department of Mathematics\\
Boise State University\\
1910 University Dr.\\
Boise ID 83725-1555}
\email{scoskey@nylogic.org}
\urladdr{boolesrings.org/scoskey}
\author{Martino Lupini}
\address{Fakult\"{a}t f\"{u}r Mathematik, Universit\"{a}t Wien,
Oskar-Morgenstern-Platz 1, Room 02.126, 1090 Wien, Austria.}
\email{martino.lupini@univie.ac.at}
\urladdr{www.lupini.org}
\thanks{The second author is supported by the York University Elia Scholars
Program and by the York University Graduate Development Fund. Much of this
work was completed during a visit of the second author to Boise State
University in April 2014. The hospitality of the BSU Mathematics Department
is gratefully acknowledged.}
\dedicatory{}
\subjclass[2000]{Primary 03C95, 03E15; Secondary 54E50}
\keywords{Model theory for metric structures, infinitary logic, Polish group
action, Urysohn sphere}

\begin{abstract}
We show that a version of L\'{o}pez-Escobar's theorem holds in the setting
of model theory for metric structures. More precisely, let $\mathbb{U}$
denote the Urysohn sphere and let $\Mod(\mathcal{L},\mathbb{U})$ be the
space of metric $\mathcal{L}$-structures supported on $\mathbb{U}$. Then for
any $\Iso(\mathbb{U})$-invariant Borel function $f\colon \Mod(\mathcal{L}, 
\mathbb{U})\rightarrow \lbrack 0,1]$, there exists a sentence $\phi $ of $%
\mathcal{L}_{\omega _{1}\omega }$ such that for all $M\in \Mod(\mathcal{L},%
\mathbb{U})$ we have $f(M)=\phi ^{M}$. This answers a question of Ivanov and
Majcher-Iwanow. We prove several consequences, for example every orbit
equivalence relation of a Polish group action is Borel isomorphic to the
isomorphism relation on the set of models of a given $\mathcal{L }%
_{\omega_{1}\omega }$-sentence that are supported on the Urysohn sphere.
This in turn provides a model-theoretic reformulation of the topological
Vaught conjecture.
\end{abstract}

\maketitle

\section{Background and statement of main result\label{Section: background
and statement}}

A well-known theorem of L\'{o}pez-Escobar \cite%
{lopez-escobar_interpolation_1965} says roughly that every Borel class of
countable structures can be axiomatized by a sentence in the logic where
countable conjunctions and disjunctions are allowed. The theorem has been
generalized to apply to wider classes of structures, using sentences from a
variety of logics (see for example \cites
{tuuri_relative_1992,vaught_invariant_1974}).

To state L\'{o}pez-Escobar's theorem more precisely, let $\mathcal{L}$ be a
countable first-order language consisting of the relational symbols $%
\{R_{i}\}$ where each $R_{i}$ has arity $n_{i}$. The space $\Mod(\mathcal{L}%
) $ of countably infinite $\mathcal{L}$-structures is given by 
\begin{equation*}
\Mod(\mathcal{L})=\prod \mathcal{P}(\mathbb{N}^{n_{i}})\;,
\end{equation*}
and we note it is compact in the product topology. The space carries a
natural $S_{\infty }$-action by left-translation on each factor, and the $%
S_{\infty }$-orbits are precisely the isomorphism classes.

Next, recall that $\mathcal{L}_{\omega _{1}\omega}$ denotes the extension of
first-order logic in which countable conjunctions and disjunctions are
allowed (formulas are still only allowed to have finitely many free
variables). If $\phi $ is a sentence of $\mathcal{L}_{\omega _{1}\omega}$
then the subset $\Mod(\phi )\subset \Mod(\mathcal{L})$ consisting just of
the models of $\phi $ is clearly $S_{\infty }$-invariant (isomorphism
invariant), and it is easy to see that it is Borel. L\'opez-Escobar's
theorem states that the converse holds, that is if $A\subset \Mod(\mathcal{L}%
)$ is Borel and $S_{\infty }$-invariant, then there exists a sentence $\phi $
of $\mathcal{L}_{\omega _{1}\omega}$ such that $A=\Mod(\phi )$.

L\'{o}pez-Escobar's theorem has numerous applications. For instance, the 
\emph{Vaught conjecture} for $\mathcal{L}_{\omega _{1}\omega }$ states that
any set $\Mod(\phi )$ contains either just countably many nonisomorphic
structures or perfectly many nonisomorphic structures (we will make this
precise in the next section). More generally, the \emph{topological Vaught
conjecture for $S_{\infty }$} states that \emph{any} Borel action of $%
S_{\infty }$ has either countably or perfectly many orbits. It follows from L%
\'{o}pez-Escobar's theorem together with some standard facts about Polish
group actions that the topological Vaught conjecture for $S_{\infty }$ is
equivalent to the Vaught conjecture for $\mathcal{L}_{\omega_{1}\omega}$.

In \cite{ivanov_polish_2013}, the authors generalize numerous properties of
the space of countable discrete structures to spaces of separable complete
metric structures. They ask whether a version of L\'opez-Escobar's theorem
holds in the metric context. In this article we confirm that the natural
generalization of L\'{o}pez-Escobar's theorem to spaces of metric structures
supported on the Urysohn sphere holds. We use this result to derive several
corollaries, including an equivalence between the topological Vaught
conjecture and a Vaught conjecture for metric structures.

Before stating our result precisely, we begin with a brief introduction to
logic for metric structures. For a full account of this fruitful area, we
refer the reader to \cite{ben_yaacov_model_2008}. As in first-order logic,
in logic for metric structures a language $\mathcal{L}$ consists of function
symbols $f$ and relation symbols $R$, each with a finite arity $n_{f}$ or $%
n_{R}$. Additionally, to each function symbol $f$ or relation symbol $R$
there is a corresponding \emph{modulus of continuity} $\varpi _{f}$ or $%
\varpi _{R}\colon \mathbb{R}_{+}\rightarrow \mathbb{R}_{+}$ which is
continuous and vanishes at $0$. Now, an $\mathcal{L}$-structure $M$ consists
of a \emph{support}, which is a complete metric space (also denoted $M$),
together with interpretations of the function and relation symbols of $%
\mathcal{L}$. That is, for each function symbol $f$ we have a function $%
f^{M}\colon M^{n_{f}}\rightarrow M$ which is uniformly continuous with
modulus of continuity $\varpi _{f}$: 
\begin{equation*}
d\left( f^{M}\left( \bar{a}\right) ,f^{M}\left( \bar{b}\right) \right) \leq
\varpi _{f}\left( d\left( \bar{a},\bar{b}\right) \right)
\end{equation*}%
(Here, as with all finite products, we consider the maximum metric on $%
M^{n_{f}}$.) Similarly, for each relation symbol $R$ we have a function $%
R^{M}\colon M^{n_{R}}\rightarrow \left[ 0,1\right] $ which is uniformly
continuous with modulus $\varpi _{R}$. 

We now briefly discuss the syntax of logic for metric structures. Given a
language $\mathcal{L}$, we define the formulas of $\mathcal{L}$ as follows.
The terms and atomic formulas are defined in the usual way, except that
instead of the $=$ symbol, we include a binary function symbol $d$ which is
always interpreted as the metric. The \emph{connectives} are continuous
functions $h\colon[0,1]^n\to[0,1]$, so if $\phi _{0},\ldots ,\phi _{n-1}$
are formulas and $h$ is such a function then $h\left( \phi _{0},\ldots ,\phi
_{n-1}\right) $ is a formula. The \emph{quantifiers} are $\sup$ and $\inf$,
so if $\phi$ is a formula and $x$ is a variable, then $\inf_{x}\phi $ and $%
\sup_{x}\phi $ are formulas.

For our generalization of L\'{o}pez-Escobar's theorem, we will use the
infinitary language $\mathcal{L}_{\omega _{1}\omega }$ in the metric setting
as defined in \cite[Theorem~1.1]{ben_yaacov_model_2009}. (Other infinitary
logics for metric structures are studied in \cite{eagle_omitting_2014}, and 
\cite{sequeira_infinitary_2013}.) Here, if $\phi _{n}$ is a sequence of $%
\mathcal{L}_{\omega _{1}\omega }$-formulas (with finitely many free
variables among them all), then $\inf_{n}\phi _{n}$ and $\sup_{n}\phi _{n}$
are $\mathcal{L}_{\omega _{1}\omega }$-formulas provided that the sequence
of uniform continuity moduli is itself uniformly bounded. Every $\mathcal{L}%
_{\omega _{1}\omega }$-formula $\phi $ has a corresponding modulus of
continuity $\varpi _{\phi }$, defined by recursion on the complexity of $%
\phi $.

Note that if $\phi$ is a sentence of $\mathcal{L}_{\omega_1\omega}$ and $M$
is an $\mathcal{L} $-structure, then $\phi^M$ is naturally interpreted as an
element of $[0,1]$. Intuitively the value $0$ means that $\phi$ is certainly
true in $M$, and larger values give shades of grey truth. Thus the
evaluation map $M\mapsto\phi^M$ is an example of a \emph{grey set}.

Grey sets, originally named graded sets, were introduced in \cite%
{ben_yaacov_grey_2014} and used extensively in \cite{ivanov_polish_2013}. If 
$X$ is a topological space then $A$ is said to be a \emph{grey subset} of $X$%
, written $A\sqsubseteq X$, if $A$ is a function $X\rightarrow \lbrack 0,1]$%
. The sets $A_{<r}=\{x\in X\mid A(x)<r\}$ and $A_{\leq r}=\{x\in X\mid
A(x)\leq r\}$ are called the \emph{level sets} of $A$. The terminology of
grey set arises from the idea that asking whether $x\in A_{<r}$ is not a
black-and-white question but rather one which depends on the parameter $r\in
\lbrack 0,1]$.

It is possible to generalize a number of concepts from point-set topology
and descriptive set theory to grey sets. For example, $A\sqsubseteq X$ is
said to be \emph{open} if $A_{<r}$ is open for all $r$ ($A$ is upper
semicontinuous), and \emph{closed} if $A_{\leq r}$ is closed for all $r$ ($A$
is lower semicontinuous); see \cite[Definition 1.4]{ben_yaacov_grey_2014}.
More generally one can define the Borel classes $\mathbf{\Sigma }_{\alpha
}^{0}$ and $\mathbf{\Pi }_{\alpha }^{0}$ of Borel grey subsets of $X$ by
induction on $\alpha \in \omega _{1}$ as in \cite[Section 2.1]%
{ivanov_polish_2013}:

\begin{itemize}
\item $A\in \mathbf{\Sigma }_{1}^{0}$ iff $A$ is an open grey subset of $X$;

\item $A\in \mathbf{\Pi }_{\alpha }^{0}$ iff $1-A\in \mathbf{\Sigma }%
_{\alpha }^{0}$; and

\item $A\in \mathbf{\Sigma }_{\alpha }^{0}$ iff $A=\inf_{n}A_{n}$ where $%
A_{n}\in \bigcup_{\beta <\alpha }\mathbf{\Pi }_{\beta }^{0}$.
\end{itemize}

We then say $A\sqsubseteq X$ is \emph{Borel} if it belongs to $\mathbf{\
\Sigma }_{\alpha }^{0}$ for some $\alpha <\omega _{1}$, and by \cite[%
Theorem~24.3]{kechris_classical_1995} $A$ is Borel iff it is Borel as a
function $X\rightarrow \left[ 0,1\right] $. Continuing up the projective
hierarchy, a grey subset $A$ of $X$ is \emph{analytic }if there is a Borel
grey subset $B\sqsubseteq X\times Y$ for some Polish space $Y$ such that $%
A=\inf_{y}B$, \emph{i.e.},\ for every $x\in X$ 
\begin{equation*}
A(x)=\inf_{y\in Y}B(x,y)\text{.}
\end{equation*}%
It is not difficult to verify that $A$ is analytic iff the level sets $%
A_{<r} $ are analytic for all $r\in \mathbb{Q}$. Similarly, $A$ is \emph{%
coanalytic} iff $1-B$ is analytic, or equivalently $B_{\leq r}$ is
coanalytic for every $r\in \mathbb{Q}$.

We now return to our motivating example of the evaluation map for a given
sentence. Fix a separable complete metric space $Y$, and denote by $\Iso(Y)$
the group of isometries of $Y$ (it is a Polish group with respect to the
topology of pointwise convergence). As with countable discrete structures,
there is naturally a space $\Mod(\mathcal{L},Y)$ of $\mathcal{L}$-structures
having $Y$ as support: 
\begin{equation*}
\Mod(\mathcal{L},Y)=\prod \mathrm{Unif}_{\varpi _{f}}(Y^{n_{f}},Y)\times
\prod \mathrm{Unif}_{\varpi _{R}}(Y^{n_{R}},[0,1])
\end{equation*}%
Here $\mathrm{Unif}_{\varpi }(A,B)$ denotes the space of $\varpi $-uniformly
continuous functions from $A$ to $B$ with the topology of pointwise
convergence. Then $\Mod(\mathcal{L},Y)$ is easily seen to be a Polish $\Iso%
(Y)$-space with respect to the natural action of $\Iso(Y)$. Now if $\phi $
is an $\mathcal{L}_{\omega _{1}\omega }$-sentence we can define the \emph{%
evaluation map} $E_{\phi }\sqsubseteq \Mod(\mathcal{L},Y)$ by 
\begin{equation*}
E_{\phi }(M)=\phi ^{M}\text{.}
\end{equation*}%
More generally if $\phi \left( \bar{x}\right) $ is an $\mathcal{L}_{\omega
_{1}\omega }$-formula with $n$-free variables we can define the evaluation
map $E_{\phi }\sqsubseteq \Mod(\mathcal{L},Y)\times Y^{n}$ by 
\begin{equation*}
E_{\phi }\left( M,u\right) =\phi ^{M}(u)\text{.}
\end{equation*}
It is not difficult to verify that the evaluation function $E_{\phi }$ for a
formula of $\mathcal{L}_{\omega _{1}\omega }$ is always Borel (see
Proposition~\ref{Proposition: grey quantifiers}).

This brings us to our main result, which asserts that \emph{any} grey subset
of $\Mod(\mathcal{L},\mathbb{U})$ which is Borel and $\Iso(\mathbb{U})$%
-invariant arises as an evaluation $E_{\phi }$. Here $\mathbb{U}$ denotes
the \emph{Urysohn sphere}, which is the unique metric space that is
separable, complete, ultrahomogeneous, with metric bounded by $1$, and which
contains an isometric copy of any other separable metric space with metric
bounded by $1$. A survey of the remarkable properties of the Urysohn sphere
can be found in \cite{melleray_geometric_2008}.

\begin{theorem}
\label{thm:main}For every $\Iso(\mathbb{U})$-invariant Borel grey subset $A$
of $\Mod(\mathcal{L},\mathbb{U})$ there exists a sentence $\phi $ of $%
\mathcal{L}_{\omega _{1}\omega }$ such that for all $M\in \Mod(\mathcal{L},%
\mathbb{U})$ we have $A(M)=\phi ^{M}$.
\end{theorem}

As an immediate consequence of Theorem~\ref{thm:main}, we obtain a L\'{o}%
pez-Escobar theorem for black-and-white sets as well. Let us say that an $%
\mathcal{L}_{\omega _{1}\omega }$-sentence $\phi $ is \emph{$\left\{
0,1\right\} $-valued} if $\phi ^{M}\in \left\{ 0,1\right\} $ for every $M\in %
\Mod(\mathcal{L},\mathbb{U})$. For such sentences $\phi $ we define $\Mod(%
\mathcal{L},\mathbb{U},\phi )$ to be the set of $M\in \Mod(\mathcal{L},%
\mathbb{U})$ such that $\phi ^{M}=0$.

\begin{corollary}
\label{Corollary: main}For every $\Iso(\mathbb{U})$-invariant Borel subset $%
A $ of $\Mod(\mathcal{L},\mathbb{U})$ there exists a sentence $\phi $ of $%
\mathcal{L}_{\omega _{1}\omega }$ such that $\phi $ is $\left\{ 0,1\right\} $%
-valued and $A=\Mod(\mathcal{L},\mathbb{U},\phi )$.
\end{corollary}

It is natural to ask whether the results hold with the Urysohn sphere
replaced by another space $Y$. We remark that our proof applies if $Y$ is
any approximately ultrahomogeneous, complete, separable metric space with a
dense sequence $p_{n}$ satisfying the property: For every $n$, the $\Iso(Y)$%
-orbit of $\left( p_{0},\ldots ,p_{n-1}\right) $ is definable in $Y^{n}$ in
the sense of \cite[Definition~9.16]{ben_yaacov_model_2008}. To see this,
note that one can use \cite[Proposition~9.19]{ben_yaacov_model_2008} to
prove a suitable modification of Lemma~\ref{Lemma: perturbation}.

After the first version of this paper had been posted on the Arxiv, a L\'{o}%
pez-Escobar theorem for metric structures structures was also announced by
Ben Yaacov, Nies, and Tsankov; see \cite{ben_yaacov_lopez-escobar_2014}.
While we work in the parametrization of $\mathcal{L}$-structures supported
on $\mathbb{U}$ considered in \cite{ivanov_polish_2013}, and for which the
question of Ivanov and Majcher-Iwanow was formulated, the authors of \cite%
{ben_yaacov_lopez-escobar_2014} consider a different parametrization of
arbitrary separable $\mathcal{L}$-structures with a distinguished countable
dense subset, which are coded by the sequences of values of all the
predicates on such a subset.

This article is organized as follows. In Section~\ref{Section: consequences}
we present several consequences of Theorem~\ref{thm:main} and Corollary~\ref%
{Corollary: main}. For example, we show that the topological Vaught
conjecture is equivalent to the natural formulation of the model-theoretic
Vaught Conjecture in the context of model theory for metric structures. In
Section~\ref{Section: stronger statement} we introduce some technical
components of the proof and state a theorem that is stronger than the main
result. Finally, in Section~\ref{Section: proof} we prove this stronger
theorem.

\textbf{Acknowledgement}. We would like to thank Ita\"{\i} Ben Yaacov,
Ilijas Farah, Bradd Hart, Luca Motto Ros, and Todor Tsankov for their
comments and suggestions on earlier drafts of this paper.

\section{Consequences of the main result\label{Section: consequences}}

In this section we show that several standard applications of L\'{o}%
pez-Escobar's theorem can be generalized to the setting of logic and model
theory for metric structures.

Our first corollary is the existence of a Scott sentence that axiomatizes a
single isomorphism class of structures (see for instance \cite[Theorem~12.1.8%
]{gao_invariant_2009}, or \cite[Theorem 4.2]{sequeira_infinitary_2013} for a
metric version). Since the orbits of a Polish group action are always Borel
(see \cite[Proposition~3.1.10]{gao_invariant_2009}) the following result is
an immediate consequence of Corollary~\ref{Corollary: main}.

\begin{corollary}
For every $\mathcal{L}$-structure $M$ in $\Mod(\mathcal{L},\mathbb{U})$
there is a sentence $\phi $ of $\mathcal{L}_{\omega _{1}\omega }$ such that $%
\phi $ is $\left\{ 0,1\right\} $-valued, and for any $N\in \Mod(\mathcal{L},%
\mathbb{\ U})$ we have: 
\begin{equation*}
\phi ^{N}=0\iff M\cong N
\end{equation*}
\end{corollary}

Next, recall that in Section~\ref{Section: background and statement} we
observed that if $\phi $ is an $\mathcal{L}_{\omega _{1}\omega }$-sentence
then the evaluation function $E_{\phi }$ is an $\Iso(\mathbb{U})$-invariant
Borel grey subset of $\Mod(\mathcal{L},\mathbb{U})$. In particular the
subspace $\Mod(\mathcal{L},\mathbb{U},\phi )$ of $\Mod(\mathcal{L},\mathbb{U}%
)$ consisting of just those $M$ with $\phi ^{M}=0$ is a standard Borel $\Iso(%
\mathbb{U})$-space. The next theorem will say that any standard Borel $\Iso(%
\mathbb{U})$-space is isomorphic to an $\Iso(\mathbb{U})$-space of this form.

First recall that if $E,F$ are equivalence relations on standard Borel
spaces $X,Y$, then $E$ is \emph{Borel reducible} to $F$ if there is a Borel
function $f\colon X\rightarrow Y$ such that for $x,x^{\prime }\in X$, 
\begin{equation*}
x\mathrel{E}x^{\prime }\iff f(x)\mathrel{F}f(x^{\prime })\text{.}
\end{equation*}%
If moreover such an $f$ can be taken to be a Borel isomorphism from $X$ to $%
Y $, then the equivalence relations $E$ and $F$ are said to be \emph{Borel
isomorphic}.

The following result implies that every orbit equivalence relation of a
Polish group action is Borel isomorphic to the isomorphism relation on some $%
\Mod(\mathcal{L},\mathbb{U},\phi )$. In the statement, we say that a
functional or relational symbol is \emph{$1$-Lipschitz} if its modulus of
continuity is (bounded above by) the function $f(t)=t$.

\begin{theorem}
\label{thm:equivariant}Let $\mathcal{L}$ be a relational countable language
for continuous logic containing $1$-Lipschitz symbols of unbounded arity.
Suppose that $G$ is a Polish group. If $X$ is a standard Borel $G$-space
then there exists an $\mathcal{L}_{\omega _{1}\omega }$-sentence $\phi $, a
continuous group monomorphism $\Phi \colon G\rightarrow \Iso(\mathbb{U})$,
and a Borel injection $f\colon X\rightarrow \Mod(\mathcal{L},\mathbb{U})$
such that:

\begin{itemize}
\item $\phi $ is $\left\{ 0,1\right\} $-valued;

\item $\mathop{\mathrm{rng}}(f)=\Mod(\mathcal{L},\mathbb{U},\phi )$;

\item $f$ maps distinct $G$-orbits into distinct $\Iso(\mathbb{U})$-orbits;
and

\item $f$ is $\Phi $-equivariant, that is, for all $x\in X$ and $g\in G$ we
have $f(gx)=\Phi (g)f(x)$.
\end{itemize}
\end{theorem}

\begin{proof}
Arguing as in the proof of \cite[Theorem 2.7.1(a)]{becker_descriptive_1996}
we can assume without loss of generality that $\mathcal{L}$ is the language
containing, for every $n\in \omega $, infinitely many $1$-Lipschitz symbols $%
\left( R_{i}^{n}\right) _{i\in \omega }$ of arity $n$. (This can be done by
replacing some $1$-Lipschitz symbols with $1$-Lipschitz symbols of higher
arity that do not depend on the extra coordinates.)

We now claim that we can suppose without loss of generality that $G=\Iso(%
\mathbb{U})$ and $X=F(G)^{\omega }$. Here, $F(G)$ denotes the space of
closed subsets of $G$ endowed with the Effros Borel structure \cite[Section
12.C]{kechris_classical_1995}, and $G$ acts coordinatewise on $X$ by the
left-shift. This claim follows from the following well-known facts:

\begin{itemize}
\item (Uspenskij \cite{uspenskij_universal_1986, uspenskij_group_1990}) $G$
is isomorphic to a closed subgroup of $\Iso(\mathbb{U})$.

\item (Mackey--Hjorth \cite[Theorem~3.5.2]{gao_invariant_2009}) If $G$ is a
closed subgroup of the Polish group $H$ then every Polish $G$-space $X$ can
be extended to a Polish $H$-space $\widetilde{X}$ in such a way that every $%
H $-orbit of $\widetilde{X}$ contains exactly one $G$-orbit of $X$.

\item (Becker--Kechris \cite[Theorem~3.3.4]{gao_invariant_2009}) If $X$ is a
Polish $G$-space then there is an equivariant embedding from $X$ into $%
F(G)^{\omega }$.
\end{itemize}

Next note that we can regard $G=\Iso(\mathbb{U})$ as a subspace of $\mathbb{U%
}^{\omega }$ by fixing a countable dense subset $\left( d_{n}\right) _{n\in
\omega }$ in $\mathbb{U}$ and identifying each $g$ with the sequence $%
(g(d_{n}))_{n\in \omega }$. Then it is easy to check that the map that sends
a closed subset $F\subset \Iso(\mathbb{U})$ to its closure $\overline{F}%
\subset \mathbb{U}^{\omega }$ is a Borel embedding of $\Iso(\mathbb{U})$%
-spaces. Hence we can suppose without loss of generality that $X=F(\mathbb{U}%
^{\omega })^{\omega }$. For each sequence $F=(F_{i})_{i\in \omega }\in X$ we
will construct an element $M_{F}\in \Mod(\mathcal{L},\mathbb{U})$ that codes 
$(F_{i})_{i\in \omega }$ as follows. First for each $i$ we define a sequence
of sets $A_{i}^{n}\subset \mathbb{U}^{n}$ by 
\begin{equation*}
A_{i}^{n}=\left\{ \,\bar{y}\in \mathbb{U}^{n}\mid \text{for every nbd $W $
of }\bar{y}\text{ we have $(W\times \mathbb{U}^{\omega })\cap F_{i}\neq
\emptyset $}\,\right\}
\end{equation*}%
It is easy to see the sets $A_{i}^{n}$ are closed. Moreover for each $i$ the
sets $A_{i}^{n}$ form the levels of a tree which codes $F_{i}$ in the sense
that $x\in F_{i}$ iff for all $n$ we have $x_{|n}\in A_{i}^{n}$. Now we
define the structure $M_{F}$ by interpreting the symbol $R_{i}^{n}$ as the
function 
\begin{equation*}
(R_{i}^{n})^{M_{F}}(\,\bar{y})=d(\,\bar{y},A_{i}^{n})\;.
\end{equation*}

It is now straightforward to verify, as in the proof of \cite[Theorem~3.6.1]%
{gao_invariant_2009}, that the function $f\colon F\mapsto M_{F}$ is a Borel
embedding of $\Iso(\mathbb{U})$-spaces from $X$ to $\Mod(\mathcal{L},\mathbb{%
\ U})$. By \cite[Corollary~15.2]{kechris_classical_1995} the range of $f$ is
a Borel subset of $\Mod(\mathcal{L},\mathbb{U})$. It therefore follows from
Corollary~\ref{Corollary: main} that there is an $\mathcal{L}_{\omega
_{1}\omega }$-sentence $\phi $ with the desired properties.
\end{proof}

A similar construction has been carried out with different methods in \cite[%
Proposition~1.3]{ivanov_polish_2013}. Theorem \ref{thm:equivariant} gives
further confirmation for the intuition that $\mathbb{U}$ and $\Iso(\mathbb{U}%
)$ play the same roles in model theory for metric structures as $\omega $
and $S_{\infty }$ play in first order model theory (for more examples see
for instance the main results of \cite{elliott_isomorphism_2013} and \cite%
{ivanov_polish_2013}).

We now give an application of Theorem~\ref{thm:equivariant} to the \emph{%
topological Vaught conjecture}, which is the assertion that for every Polish
group $G$ and standard Borel $G$-space $X$, either $X$ has just countably
many orbits or it has perfectly many orbits (see \cite[Section 6.2]%
{becker_descriptive_1996}). Here, $X$ is said to have perfectly many orbits
if there is an injective Borel reduction from the equality relation of $%
\mathbb{R}$ to the orbit equivalence relation of $G $ on $X$. In the
following result, the implication (1)$\Rightarrow $(2) is obvious, and (2)$%
\Rightarrow $(1) is an immediate consequence of Theorem~\ref{thm:equivariant}%
.

\begin{corollary}
\label{Corollary: Vaught conjecture}Let $\mathcal{L}$ denote a relational
countable language for continuous logic containing $1$-Lipschitz symbols of
unbounded arity. The following statements are equivalent:

\begin{enumerate}
\item The topological Vaught conjecture holds;

\item If $\phi $ is any $\mathcal{L}_{\omega _{1}\omega }$-sentence then $%
\Mod(\mathcal{L},\mathbb{U},\phi )$ has either countably many or perfectly
many isomorphism classes.
\end{enumerate}
\end{corollary}

Finally, we consider an application to infinitary logic. In L\'{o}%
pez-Escobar's original work, he was interested foremost in establishing an
interpolation property for the logic $\mathcal{L}_{\omega _{1}\omega }$.
What we have called L\'{o}pez-Escobar's theorem is in fact equivalent to
this interpolation result. We will now show that a similar phenomenon holds
in the setting of logic for metric structures. First we need an analog of
the Luzin separation theorem \cite[Theorem~1.6.1]{gao_invariant_2009} for
grey sets.

\begin{proposition}
\label{Proposition: luzin}

\begin{enumerate}
\item Let $X$ be a Polish space, and suppose that $A,B$ are grey subsets of $%
X$, $A$ is analytic, $B$ is coanalytic, and $A\geq B$. Then there is a Borel
grey subset $C\sqsubseteq X$ such that $A\geq C\geq B$.

\item Let $X$ be a Polish $G$-space, $A,B$ as above, and suppose
additionally that $A,B$ are $G$-invariant. Then there is a $G$-invariant
Borel grey subset $C\sqsubseteq X$ such that $A\geq C\geq B$.
\end{enumerate}
\end{proposition}

\begin{proof}
(a) Fix $r\in \mathbb{Q}\cap \left[ 0,1\right] $ and observe that $%
A_{<r}\subset B_{\leq r}$, where $A_{<r}$ is analytic and $B_{\leq r}$ is
coanalytic. Therefore by \cite[Theorem~1.6.1]{gao_invariant_2009} there is $%
P^{\left( r\right) }\subset X$ Borel such that $A_{<r}\subset P^{\left(
r\right) }\subset B_{\leq r}$. Now let $C\sqsubset X$ be the grey subset
defined by 
\begin{equation*}
C(x)=\inf \left\{ r\in \mathbb{Q}\cap \left[ 0,1\right] \mid x\in P^{\left(
r\right) }\right\} \text{.}
\end{equation*}
It is straightforward to verify that $C$ is Borel and $A\geq C\geq B$.

(b) By part (a) there is a Borel grey subset $D$ of $X$ such that $A\geq
D\geq B$. Define the grey subset $C\ $of $X$ by $C(x)\leq r$ if and only if $%
\forall ^{\ast }g\in G$, $D(gx)\leq r$. It is not difficult to verify by
induction on the Borel rank of $D$ that $C$ is a Borel $G$-invariant subset
of $X$ (see also Proposition \ref{Proposition: grey quantifiers}). It is
clear that $A\geq C\geq B$, which concludes the proof.
\end{proof}

We are now ready to prove the interpolation theorem for $\mathcal{L}_{\omega
_{1}\omega }$. In the following if $\mathcal{L}$ and $\mathcal{R}$ are
possibly distinct languages, and $\phi $ is an $\mathcal{L}_{\omega
_{1}\omega }$-sentence, and $\rho $ is an $\mathcal{R}_{\omega _{1}\omega }$%
-sentence, then we write $\phi \models \rho $ iff $\phi ^{M}\geq \rho ^{M}$
for every $M\in \Mod(\mathcal{L}\cup \mathcal{R},\mathbb{U})$.

\begin{corollary}
\label{Corollary: interpolation}The logic $\mathcal{L}_{\omega _{1}\omega }$
has the following interpolation property: Suppose that $\mathcal{L},\mathcal{%
R},\mathcal{S}$ are pairwise disjoint countable languages, $\phi $ is a
sentence in $\left( \mathcal{L}\cup \mathcal{R}\right) _{\omega _{1}\omega }$
and $\rho $ is a sentence in $\left( \mathcal{L}\cup \mathcal{S}\right)
_{\omega _{1}\omega }$. If $\phi \models \rho $, then there is an $\mathcal{L%
}_{\omega _{1}\omega }$-sentence $\tau $ such that $\phi \models \tau $ and $%
\tau \models \rho $.
\end{corollary}

\begin{proof}
We can canonically identify $\Mod(\mathcal{L}\cup \mathcal{S},\mathbb{U})$
with $\Mod(\mathcal{L},\mathbb{U})\times \Mod(\mathcal{S},\mathbb{U})$ and $%
\Mod(\mathcal{L}\cup \mathcal{R},\mathbb{U})$ with $\Mod(\mathcal{L},\mathbb{%
\ U})\times \Mod(\mathcal{R},\mathbb{U})$. Define $A$ to be the analytic
subset of $\Mod(\mathcal{L},\mathbb{U})$ 
\begin{equation*}
A=\inf\nolimits_{M\in \Mod(\mathcal{S},\mathbb{U})}E_{\phi }
\end{equation*}%
where $E_{\phi }\sqsubseteq \Mod(\mathcal{L},\mathbb{U})\times \Mod(\mathcal{%
\ S},\mathbb{U})$. Similarly define $B$ to be the coanalytic subset of $\Mod(%
\mathcal{L},\mathbb{U})$ 
\begin{equation*}
B=\sup\nolimits_{M\in \Mod(\mathcal{S},\mathbb{U})}E_{\rho }\text{.}
\end{equation*}%
Observe that $A\geq B$ since $\phi \models \rho $. Therefore by Proposition~%
\ref{Proposition: luzin} there is a $\Iso(\mathbb{U})$-invariant Borel grey
subset $C$ of $\Mod(\mathcal{L},\mathbb{U})$ such that $A\geq C\geq B$. By
Theorem \ref{thm:main} there is an $\mathcal{L}_{\omega _{1}\omega }$%
-sentence $\tau $ such that $C=E_{\tau }$. It is immediate to verify that $%
A\geq E_{\tau }\geq B$ implies that $\phi \models \tau $ and $\tau \models
\rho $.
\end{proof}

\section{Further notions and a strengthening of the main result\label%
{Section: stronger statement}}

In this section we formulate a statement that is stronger than Theorem~\ref%
{thm:main} and handles the case when $A$ is a grey subset that is not
invariant. Our motivation for this proof strategy comes from Vaught's
dynamical proof of L\'{o}pez-Escobar's theorem (see \cite%
{vaught_invariant_1974} or \cite[Theorem 7.8]{kechris_classical_1995}).

In order to state the stronger result, we will need to introduce the
following category quantifiers for grey sets. These generalize the classical
category quantifiers $\exists ^{\ast }$ and $\forall ^{\ast }$ as defined
for instance in \cite[Section~8.J]{kechris_classical_1995}. If $X,Y$ are
Polish spaces, $U\subset Y$ is open, and $B$ is a grey subset of $X\times Y$%
, then we define the grey subsets $\inf_{y\in U}^*B$ and $\sup_{y\in U}^*B$
of $X$ by the properties: 
\begin{align*}
\left( \inf\nolimits_{y\in U}^{\ast }B\right) (x)<r& \iff \exists ^{\ast
}y\in U\text{ such that }B(x,y)<r\text{,} \\
\left( \sup\nolimits_{y\in U}^{\ast }B\right) (x)>r& \iff \exists ^{\ast
}y\in U\text{ such that }B(x,y)>r\text{.}
\end{align*}

The next proposition lists some of the basic properties of these
set-theoretic category quantifiers. They can be proved with the same
arguments as Propositions~3.2.5, 3.2.6, and Theorem~3.2.7 of \cite%
{gao_invariant_2009} (in \cite[3.2.5-3.2.7]{gao_invariant_2009} the space $Y$
of Proposition \ref{Proposition: grey quantifiers} appears as a Polish group 
$G$ acting on $X$). Note that in the statement, as in the rest of this
article, all the usual arithmetic operations in fact denote their \emph{%
truncated} versions to the interval $[0,1]$. For example if $a,b\in \mathbb{R%
}$ then $a+b$ stands for 
\begin{equation*}
\max \left\{ 0,\min \left\{ 1,a+b\right\} \right\}
\end{equation*}%
and similarly for the other operations.

\begin{proposition}
\label{Proposition: grey quantifiers}Let $X,Y$ be Polish, $U\subset Y$ open,
and $B$ a grey subset of $X\times Y$.

\begin{enumerate}
\item $\inf_{y\in U}^{\ast }\left( q-B\right) =q-\sup_{y\in U}^{\ast }B$ for
any $q\in \left[ 0,1\right] $;\label{Clause: negation grey quantifier}


\item $\sup_{y\in U}^{\ast }B=\sup_{n\in \omega }\inf_{y\in W_{n}\cap
U}^{\ast }B$, where $W_{n}$ enumerates a basis for $Y$;

\item If $B_{n}$ is a sequence of grey subsets of $X\times Y$, then $%
\inf_{n}\inf_{y}^{\ast }B_{n}=\inf_{y}^{\ast }\inf_{n}B_{n}$ and $\sup_{y\in
U}^{\ast }\sup_{n}B_{n}=\sup_{n}\sup_{y\in U}^{\ast }B_{n}$;

\item If $B$ is open then $\inf_{y\in U}^{\ast }B$ is open;

\item If $B$ is $\mathbf{\Sigma }_{\alpha }^{0}$ then $\inf_{y\in U}^{\ast
}B $ is $\mathbf{\Sigma }_{\alpha }^{0}$;

\item If $B$ is $\mathbf{\Pi }_{\alpha }^{0}$ then $\sup_{y\in U}^{\ast }B$
is $\mathbf{\Pi }_{\alpha }^{0}$.
\end{enumerate}
\end{proposition}


Although we will refrain from using the notation in our proof, it is worth
remarking that the category quantifiers can be used to define a version of
the Vaught transforms in the grey setting. (The grey Vaught transforms were
first introduced in \cite[Section~2.1]{ivanov_polish_2013}.) If $X$ is a
Polish $G$-space, $A\sqsubseteq X$ is Borel, and $U\sqsubseteq G$ is open,
then 
\begin{align*}
A^{\ast U}(x)& =\sup\nolimits_{g\in G}^{\ast }\left( A(gx)-U(g)\right) \text{%
, and} \\
A^{\triangle U}(x)& =\inf\nolimits_{g\in G}^{\ast }\left( A(gx)+U(g)\right) 
\text{.}
\end{align*}%
The basic properties of the Vaught transforms listed in \cite[Lemma 2.4]%
{ivanov_polish_2013} can easily be obtained as a consequence of Proposition %
\ref{Proposition: grey quantifiers}.


We will also need some notation for a family of \textquotedblleft
basic\textquotedblright\ open graded subsets of $\Iso(\mathbb{U})$. We fix
once and for all an enumeration $\mathbf{p}=\left( p_{n}\right) _{n\in
\omega }$ of a dense subset of $\mathbb{U}$. For any $u\in \mathbb{U}^{k}$
we define the open grey subset $[u]$ of $\Iso(\mathbb{U})$ by 
\begin{equation*}
\left[ u\right] (g)=d(g^{-1}\mathbf{p}_{|k},u)\text{.}
\end{equation*}%
(Here as usual $d$ denotes the maximum metric on $\mathbb{U}^{k}$.) We also
let $\mathcal{O}(\mathbf{p}_{|k})$ be the orbit of $\mathbf{p}_{|k}$ under
the action of $\mathrm{Iso}(\mathbb{U})$, which coincides with the set of
realizations of the type of the $k$-tuple $\mathbf{p}_{|k}$. The level sets $%
[u]_{<r}$, where $u\in \mathcal{O}(\mathbf{p}_{|k})$ and $r>0$, form an open
basis for the topology of $\Iso(\mathbb{U})$.

We are now ready to state our strengthening of Theorem~\ref{thm:main}.
Roughly speaking, the result accommodates Borel graded sets that are not
invariant, at the cost of taking a Vaught transform and allowing parameters
in the formula $\phi $. In the statement, we say that a formula $\phi$ is 
\emph{$N$-Lipschitz} if its modulus of continuity is bounded above by the
function $f(t)=Nt$.

\begin{theorem}
\label{thm:parameters}Suppose that $\mathcal{L}$ is a countable language for
continuous logic, $\mathbf{p}$ is as above, and $k\in \mathbb{N}$. For any
Borel grey subset $A\sqsubseteq \Mod(\mathcal{L},\mathbb{U})$ and for any $%
N\in \mathbb{N}$ there exists an $N$-Lipschitz $\mathcal{L}_{\omega
_{1}\omega }$-formula $\phi $ with $k$ free variables such that for every $%
M\in \Mod(\mathcal{L},\mathbb{U})$ and $u\in \mathbb{U}^{k}$, we have 
\begin{equation*}
\sup\nolimits_{g\in \Iso(\mathbb{U})}^{\ast }[A(gM)-Nd(g^{-1}\mathbf{p}%
_{|k},u)]=\phi ^{M}(u)\text{.}
\end{equation*}
\end{theorem}

Theorem~\ref{thm:main} follows as the special case when $k=0$ and $N=1$.
Indeed, if $A$ is an $\Iso(\mathbb{U})$-invariant grey subset of $\Mod(%
\mathcal{L},\mathbb{U})$, then%
\begin{equation*}
A(M)=\sup\nolimits_{g\in \Iso(\mathbb{U})}^{\ast }A(gM)\text{.}
\end{equation*}%
Therefore Theorem~\ref{thm:parameters} yields a sentence $\phi $ such that $%
A(M)=\phi ^{M}$ for $M\in \Mod(\mathcal{L},\mathbb{U})$.

In the proof of the theorem we will need the following perturbation result,
which is similar to \cite[Lemma 2.3]{ivanov_polish_2013}. In the statement,
we denote by $\tau _{k}\left( \bar{x},\bar{y}\right) $ the quantifier-free
formula with $2k$ free variables given by 
\begin{equation*}
\max_{i,j\in k}\left\vert d\left( x_{i},x_{j}\right) -d\left(
y_{i},y_{j}\right) \right\vert \text{.}
\end{equation*}%
Observe that $\tau (\bar{x},\mathbf{p}_{|k})$ can be regarded as a
quantifier-free formula with $k$ variables.

\begin{lemma}
\label{Lemma: perturbation} For all $\varepsilon >0$, if $u,w\in \mathbb{U}%
^{k}$ are such that $\tau _{k}\left( u,w\right) <\varepsilon $, then there
is $g\in \Iso(Y)$ such that $d(u,gw)<3\varepsilon $.
\end{lemma}

\begin{proof}
Consider the metric space $Z$ obtained from the disjoint union of $\left\{
u_{i}:i\in k\right\} $ and $\left\{ w_{i}:i\in k\right\} $ as in \cite[%
Example 56]{petersen_riemannian_2006}, where%
\begin{equation*}
d\left( u_{i},w_{j}\right) =\min_{n\in k}\left( d(u_{i},u_{n})+\varepsilon
+d(w_{n},w_{j})\right) \text{.}
\end{equation*}%
By the finite injectivity of Urysohn space \cite{melleray_geometry_2007} the
isometric embedding of $\left\{ u_{i}:i\in k\right\} $ in $\mathbb{U}$
extends to an isometric embedding of $Z$ into $\mathbb{U}$. This gives $%
\widetilde{w}=\left( \widetilde{w}_{j}\right) _{j\in k}\in \mathbb{U}^{k}$
such that%
\begin{align*}
d(\widetilde{w},u)& <3\varepsilon \text{, and} \\
d(\widetilde{w}_{i},\widetilde{w}_{j})& =d(w_{i},w_{j})
\end{align*}%
for $i,j\in k$. Since $\mathbb{U}$ is ultrahomogeneous, there is an isometry 
$g\in \Iso(\mathbb{U})$ such that $gw=\widetilde{w}$ and hence $%
d(gw,u)<3\varepsilon $.
\end{proof}


We remark that Lemma~\ref{Lemma: perturbation} together with \cite[Proposition~9.19]{ben_yaacov_model_2008} implies that $\mathcal{O}(\mathbf{p}_{|k})$ is a definable subset of $\mathbb{U}^{k}$ in the sense of \cite[Definition~9.16]{ben_yaacov_model_2008}.

\section{The proof\label{Section: proof}}

\label{Section: the proof}

In this section, we give the proof of Theorem~\ref{thm:parameters}. To
begin, we let $\mathcal{B}$ denote the family of Borel grey subsets of $\Mod(%
\mathcal{L},\mathbb{U})$ which satisfy the conclusion of Theorem~\ref%
{thm:parameters} for all $k\in \mathbb{N}$. Our strategy will be to show
that $\mathcal{B}$ has the following properties.

\begin{enumerate}
\item \label{Clause: negation}If $A\in \mathcal{B}$ then $q-A\in \mathcal{B}$
for every $q\in \left[ 0,1\right] $ (Section~\ref{Subsection: negation});

\item \label{Clause: base case}For every $n\in \mathbb{N}$ and every
quantifier-free $\mathcal{L}$$_{\omega \omega }$-formula $\phi \left( \bar{x}%
\right) $ with $n$ free variables the grey subset $E_{\phi ,\mathbf{p}_{|n}}$
of $\Mod\left( \mathcal{L},Y\right) $ defined by 
\begin{equation*}
E_{\phi ,\mathbf{p}_{|n}}(M)=\phi ^{M}(\mathbf{p}_{|n})
\end{equation*}%
is in $\mathcal{B}$ (Section~\ref{Subsection: base case});

\item \label{Clause: linear combination}If $A,B\in \mathcal{B}$ and $\lambda
,\mu \in \left[ 0,1\right] $ then $\lambda A+\mu B\in \mathcal{B}$ (Section~%
\ref{Subsection: linear combinations});

\item \label{Clause: sup and inf}If $A_{n}\in \mathcal{B}$ for every $n\in
\omega $, then $\inf_{n}A_{n}\in \mathcal{B}$ and $\sup_{n}A_{n}\in \mathcal{%
\ B}$ (Section~\ref{Subsection: infima and suprema}).
\end{enumerate}

We once again remind the reader that in (3), as everywhere, the arithmetic
operations denote their truncated versions.

We now show that these facts ensure that the family $\mathcal{B}$ contains
all Borel grey subsets of $\Mod(\mathcal{L},\mathbb{U})$. For this we need
the following lemma. In the statement, recall that a family of functions 
\emph{separates the points} of $X$ if for every distinct $x,y\in X$ there is 
$f$ in the family such that $f(x)\neq f(y)$.

\begin{lemma}
\label{Lemma: class Borel grey sets}Suppose that $X$ is a standard Borel
space, $\mathcal{F}$ is a family of Borel grey sets of $X$, and $\mathcal{F}%
_{0}\subset\mathcal{F}$ is a countable subfamily that separates the points
of $X$. Assume further that $\mathcal{F}$ satisfies the following closure
properties:

\begin{enumerate}
\item \label{Clause: negation F}If $A\in \mathcal{F}$ then $q-A\in \mathcal{F%
}$ for every $q\in \left[ 0,1\right] $;

\item \label{Clause: constants F}Every constant function belongs to $%
\mathcal{F}$.

\item \label{Clause: linear combination F}If $A,B\in \mathcal{F}$ and $%
\lambda ,\mu \in \left[ 0,1\right] $ then $\lambda A+\mu B\in \mathcal{F}$;

\item \label{Clause: sup and inf F}If $A_{n}\in \mathcal{F}$ for every $n\in
\omega $, then $\inf_{n}A_{n}\in \mathcal{F}$ and $\sup_{n}A_{n}\in \mathcal{%
\ F}$;
\end{enumerate}

Then $\mathcal{F}$ contains all Borel grey sets.
\end{lemma}

\begin{proof}
By induction it follows from \eqref{Clause: linear combination F} that $%
\mathcal{F}$ is closed under arbitrary finite linear combinations with
coefficients in $\left[ 0,1\right] $. Moreover one can deduce from %
\eqref{Clause: sup and inf F} that $\mathcal{F}$ is closed under pointwise
limits. Arguing as in the proof of \cite[Theorem 11.6]%
{kechris_classical_1995} one can show that any Borel grey set is a pointwise
limit of linear combinations of $\left\{ 0,1\right\} $-valued Borel grey
sets. Therefore it is enough to show that for every Borel subset $U$ of $X$,
the zero-indicator $\mathbf{0}_{U}$ of $U$ lies in $\mathcal{F}$. (Here the
zero indicator $\mathbf{0}_{U}$ is the function constantly equal to $0$ on $%
U $ and constantly equal to $1$ on $X\smallsetminus U$; see \cite[%
Notation~1.2]{ben_yaacov_grey_2014}.)

For this, let $\mathcal{U}$ denote the family of Borel subsets $U$ of $X$
such that $\mathbf{0}_{U}\in \mathcal{F}$. Also let $\mathcal{U}_{0}$ denote
the family of level sets $A_{\leq q}$ for $A\in \mathcal{F}_{0}$ and $q\in 
\mathbb{Q}\cap \left[ 0,1\right] $. It follows from 
\eqref{Clause: negation
F} and \eqref{Clause: sup and inf F} that $\mathcal{U}$ is a $\sigma $%
-algebra of Borel subsets of $X$. Moreover since $\mathcal{F}_{0}$ separates
the points of $X$, $\mathcal{U}_{0}$ is a countable family of Borel sets
that separate the points of $X$. By \cite[Theorem 3.3]{mackey_borel_1957} in
order to show that $\mathcal{U}$ contains all Borel sets it is enough to
prove that $\mathcal{U}_{0}$ is contained in $\mathcal{U}$. For this,
observe that for each $A\in \mathcal{F}_{0}$ and $q\in \mathbb{Q}\cap \left[
0,1\right] $ the indicator function $\mathbf{0}_{A_{\leq q}}$ of the level
set $A_{\leq q}$ is $\sup_{m\in \mathbb{N}}m\left( A-q\right) $. By %
\eqref{Clause: constants F}, and \eqref{Clause: linear combination F} we
have $m\left( A-q\right) \in \mathcal{F}$ for every $m\in \mathbb{N}$ and
hence $\mathbf{0}_{A_{\leq q}}\in \mathcal{F}$ by \eqref{Clause: sup and inf}%
. Therefore $A_{\leq q}\in \mathcal{U}$, as claimed.
\end{proof}

We may now give the conclusion of the proof of the main theorem.

\begin{proof}[Proof of Theorem~\protect\ref{thm:parameters}]
By Lemmas~\ref{Lemma: negation in B}, \ref{Lemma: base case}, \ref{Lemma:
linear combinations in B}, and \ref{Lemma: sup and inf in B} below, the
family $\mathcal{B}$ of grey sets satisfying the conclusion of the theorem
satisfies hypotheses (1), (2), (3), (4) of Lemma~\ref{Lemma: class Borel
grey sets}. Let $\mathcal{B}_{0}$ denote the family of grey subsets of $\Mod(%
\mathcal{L},\mathbb{U})$ of the form 
\begin{align*}
M& \mapsto R^{M}\left( p_{i_{0}},\ldots ,p_{i_{n-1}}\right) \text{, or} \\
M& \mapsto d\left( f^{M}\left( p_{i_{0}},\ldots ,p_{i_{n-1}}\right)
,p_{i_{n}}\right)
\end{align*}%
where $i_{0},\ldots ,i_{n}\in \mathbb{N}$ and $f,R$ are $n$-ary symbols of $%
\mathcal{L}$. It is straightforward to verify that $\mathcal{B}_{0}$
separates the points of $\Mod(\mathcal{L},\mathbb{U})$. Moreover, by Lemma~%
\ref{Lemma: base case}, $\mathcal{B}_{0}$ is contained in $\mathcal{B}$. It
therefore follows from Lemma \ref{Lemma: class Borel grey sets} that $%
\mathcal{B}$ contains all Borel grey sets, as desired.
\end{proof}

We now proceed to verify each of the closure properties outlined at the
beginning of this section.

\subsection{Negation\label{Subsection: negation}}

Recall that $\mathcal{O}(\mathbf{p}_{|k})$ for $k\in \omega $ denotes the
orbit of $\mathbf{p}_{|k}$ under the action $\Iso(\mathbb{U}%
)\curvearrowright \mathbb{U}^{k}$.

\begin{lemma}
\label{Lemma: negation transform}Suppose that $A$ is a grey subset of $\Mod(%
\mathcal{L},\mathbb{U})$, $k,N\in \omega $ with $N\geq 1$, and $u\in \mathbb{%
U}^{k}$. For any $t\in \left[ 0,1\right] $, the following statements are
equivalent:

\begin{enumerate}
\item $\inf\nolimits_{g\in \Iso(\mathbb{U})}^{\ast }[A(gM)+Nd(u,g^{-1}%
\mathbf{p}_{|k})]<t$;

\item there are $\widetilde{k}\geq k$, $\widetilde{N}\geq N$, and $%
\widetilde{u}\in \mathcal{O}(\mathbf{p}_{|\widetilde{k}})$ such that%
\begin{equation*}
Nd(\widetilde{u}_{|k},u)+\sup\nolimits_{g\in \Iso(\mathbb{U})}^{\ast }[A(gM)-%
\widetilde{N}d(\widetilde{u},g^{-1}\mathbf{p}_{|\widetilde{k}})]<t\text{;%
\label{Eqn: star and triangle}}
\end{equation*}

\item there are $t_{0}<t$, $\widetilde{k}\geq k$, $\widetilde{N}\geq N$,
such that for every $m\geq 1$ there is $\widetilde{u}\in \mathbb{U}^{%
\widetilde{k}}$ such that%
\begin{equation*}
\qquad m\tau _{\widetilde{k}}(\widetilde{u},\mathbf{p}_{|\widetilde{k}%
})+Nd\left( \widetilde{u}_{|k},u\right) +\sup\nolimits_{g\in \Iso(\mathbb{U}%
)}^{\ast }[A(gM)-\widetilde{N}d(\widetilde{u},g^{-1}\mathbf{p}_{|\widetilde{k%
}})]<t_{0}\text{.}
\end{equation*}
\end{enumerate}
\end{lemma}

\begin{proof}
(1)$\Rightarrow $(2) Suppose that 
\begin{equation*}
\inf\nolimits_{g\in \Iso(\mathbb{U})}^{\ast }[A(gM)+Nd(u,g^{-1}\mathbf{p}%
_{|k})]<t.
\end{equation*}%
Thus there are $s,r\in \left[ 0,1\right] $ such that $s+r<t$ and $\exists
^{\ast }g\in \Iso(\mathbb{U})$ such that $A(gM)<r$ and $Nd(u,g^{-1}\mathbf{p}%
_{|k})<s$. In particular there is a nonempty open $U\subset \left[ u\right]
_{<sN^{-1}}$ such that $\forall ^{\ast }g\in U$, $A(gM)<r$. Pick $g_{0}\in U$
and observe that $Nd(u,g_{0}^{-1}\mathbf{p}_{|k})<s$. Define $\widetilde{k}%
\geq k$ and $\widetilde{N}\geq N$ such that if $g\in \Iso(\mathbb{U})$ is
such that $\widetilde{N}d(g^{-1}\mathbf{p}_{|\widetilde{k}},g_{0}^{-1}%
\mathbf{p}_{|\widetilde{k}})<1$ then $g\in U$. Define $\widetilde{u}%
=g_{0}^{-1}\mathbf{p}_{|\widetilde{k}}\in \mathcal{O}(\mathbf{p}_{|%
\widetilde{k}})$. Observe that%
\begin{equation*}
Nd\left( \widetilde{u}_{|k},u\right) =Nd(g_{0}^{-1}\mathbf{p}_{|k},u)<s\text{%
.}
\end{equation*}%
Moreover $\forall ^{\ast }g\in \left[ \widetilde{u}\right] _{<\widetilde{N}%
^{-1}}$, 
\begin{equation*}
A(gM)<r\leq r+\widetilde{N}d(\widetilde{u},g^{-1}\mathbf{p}_{|\widetilde{k}})
\end{equation*}%
Therefore%
\begin{equation*}
\sup\nolimits_{g\in \Iso(\mathbb{U})}^{\ast }[A(gM)-\widetilde{N}d(%
\widetilde{u},g^{-1}\mathbf{p}_{|\widetilde{k}})]\leq r
\end{equation*}%
and hence%
\begin{equation*}
Nd\left( \widetilde{u}_{|k},u\right) +\sup\nolimits_{g\in \Iso(\mathbb{U}%
)}^{\ast }[A(gM)-\widetilde{N}d(\widetilde{u},g^{-1}\mathbf{p}_{|\widetilde{k%
}})]\leq r+s<t\text{.}
\end{equation*}

(2)$\Rightarrow $(3) This is obvious, since $\widetilde{u}\in \mathcal{O}(%
\mathbf{p}_{|k})$ implies $\tau _{\widetilde{k}}(\widetilde{u},\mathbf{p}_{|%
\widetilde{k}})=0$.

(3)$\Rightarrow $(2) By hypothesis there are $t_{0}<t$, $\widetilde{k}\geq k$%
, $\widetilde{N}\geq N$, such that for every $m\geq 1$ there is $\widetilde{u%
}\in \mathbb{U}^{\widetilde{k}}$ such that%
\begin{equation*}
m\tau _{\widetilde{k}}\left( \widetilde{u},\mathbf{p}_{|\widetilde{k}%
}\right) +Nd\left( \widetilde{u}_{|k},u\right) +\sup\nolimits_{g\in \Iso(%
\mathbb{U})}^{\ast }[A(gM)-\widetilde{N}d(\widetilde{u},g^{-1}\mathbf{p}_{|%
\widetilde{k}})]<t_{0}\text{.}
\end{equation*}%
Fix $m\in \mathbb{N}$ such that $t_{0}+\frac{6\widetilde{N}}{m}<t$. Let $%
\widetilde{u}\in \mathbb{U}^{\widetilde{k}}$ be such that%
\begin{equation*}
m\tau _{\widetilde{k}}(\widetilde{u},\mathbf{p}_{|\widetilde{k}})+Nd\left( 
\widetilde{u}_{|k},u\right) +\sup\nolimits_{g\in \Iso(\mathbb{U})}^{\ast
}[A(gM)-\widetilde{N}d(\widetilde{u},g^{-1}\mathbf{p}_{|\widetilde{k}%
})]<t_{0}\text{.}
\end{equation*}%
Since $\tau _{\widetilde{k}}(\widetilde{u},\mathbf{p}_{|\widetilde{k}})<1/m$%
, by Lemma \ref{Lemma: perturbation} there is $g_{0}\in \Iso(\mathbb{U})$
such that $d(\widetilde{v},\widetilde{u})<3/m$, where $\widetilde{v}%
=g_{0}^{-1}\mathbf{p}_{|\widetilde{k}}\in \mathcal{O}(\mathbf{p}_{|%
\widetilde{k}})$. Observe now that%
\begin{eqnarray*}
&&Nd\left( \widetilde{v}_{|k},u\right) +\sup\nolimits_{g\in \Iso\mathbb{(U}%
)}^{\ast }[A(gM)-\widetilde{N}d(\widetilde{v},g^{-1}\mathbf{p}_{|\widetilde{k%
}})] \\
&\leq &Nd\left( \widetilde{u}_{|k},u\right) +\sup\nolimits_{g\in \Iso(%
\mathbb{U})}^{\ast }[A(gM)-\widetilde{N}d(\widetilde{u},g^{-1}\mathbf{p}_{|%
\widetilde{k}})]+2\widetilde{N}d\left( \widetilde{v},\widetilde{u}\right) \\
&\leq &t_{0}+\frac{6\widetilde{N}}{m}<t\text{.}
\end{eqnarray*}

(2)$\Rightarrow $(1) By hypothesis there are $\widetilde{k}\geq k$, $%
\widetilde{N}\geq N$, and $\widetilde{u}\in \mathcal{O}(\mathbf{p}_{|%
\widetilde{k}})$ such that%
\begin{equation*}
Nd\left( \widetilde{u}_{|k},u\right) +\sup\nolimits_{g\in \Iso(\mathbb{U}%
)}^{\ast }[A(gM)-\widetilde{N}d(\widetilde{u},g^{-1}\mathbf{p}_{|\widetilde{k%
}})]<t\text{.}
\end{equation*}%
Define $s=Nd\left( \widetilde{u}_{|k},u\right) $ and 
\begin{equation*}
r=\sup\nolimits_{g\in \Iso(\mathbb{U})}^{\ast }[A(gM)-\widetilde{N}d(%
\widetilde{u},g^{-1}\mathbf{p}_{|\widetilde{k}})]\text{.}
\end{equation*}%
Fix $\delta >0$ such that $s+r+2\delta <t$. Observe that since $\widetilde{u}%
\in \mathcal{O}(\mathbf{p}_{\widetilde{k}})$ we have that $\left[ \widetilde{%
u}\right] _{<\delta \widetilde{N}^{-1}}\neq \varnothing $. Moreover $\left[ 
\widetilde{u}\right] _{<\delta \widetilde{N}^{-1}}\subset \left[ u\right]
_{<\left( s+\delta \right) N^{-1}}$. In fact suppose that $g\in \left[ 
\widetilde{u}\right] _{<\delta \widetilde{N}^{-1}}$ and hence $\widetilde{N}%
d(g^{-1}\mathbf{p}_{|\widetilde{k}},\widetilde{u})<\delta $. Thus we have%
\begin{eqnarray*}
Nd(g^{-1}\mathbf{p}_{|k},u) &\leq &Nd\left( \widetilde{u}_{|k},u\right) +Nd(%
\widetilde{u}_{|k},g^{-1}\mathbf{p}_{|k}) \\
&\leq &s+\widetilde{N}d(\widetilde{u},g^{-1}\mathbf{p}_{|\widetilde{k}}) \\
&\leq &s+\delta \text{.}
\end{eqnarray*}%
Moreover have that $\forall ^{\ast }g\in \left[ \widetilde{u}\right]
_{<\delta \widetilde{N}^{-1}}$,%
\begin{equation*}
A\left( gM\right) \leq r+\widetilde{N}d(\widetilde{u},g^{-1}\mathbf{p}_{|%
\widetilde{k}})<r+\delta \text{.}
\end{equation*}%
It follows that $\left[ u\right] _{<\left( s+\delta \right) N^{-1}}\neq
\varnothing $ and $\exists ^{\ast }g\in \left[ u\right] _{<\left( s+\delta
\right) N^{-1}}$ such that $A\left( gM\right) <r+\delta $. Therefore%
\begin{equation*}
\inf\nolimits_{g\in \Iso(\mathbb{U})}^{\ast }[A(gM)+Nd(u,g^{-1}\mathbf{p}%
_{|k})]\leq s+\delta +r+\delta <t\text{.}
\end{equation*}%
This concludes the proof.
\end{proof}

\begin{lemma}
\label{Lemma: inf}If $A$ is a grey subset of $\Mod(\mathcal{L},\mathbb{U})$,
then $A\in \mathcal{B}$ if and only if for every $k,N\in \omega $ with $%
N\geq 1$ there is an $N$-Lipschitz formula $\varphi $ such that, for every $%
M\in \Mod(\mathcal{L},\mathbb{U})$,%
\begin{equation}
\inf\nolimits_{g\in \Iso(\mathbb{U})}^{\ast }[A(gM)+Nd(g^{-1}\mathbf{p}%
_{|k},u)]=\varphi ^{M}(u)\text{.\label{Equation:inf}}
\end{equation}
\end{lemma}

\begin{proof}
Suppose that $A\in \mathcal{B}$. We have that for every $\widetilde{k},%
\widetilde{N}\in \omega $ such that $\widetilde{N}\geq 1$ there is a formula 
$\psi _{\widetilde{k},\widetilde{N}}$ in $\widetilde{k}$ free variables such
that%
\begin{equation*}
\sup\nolimits_{g\in \Iso(\mathbb{U})}^{\ast }[A(gM)-\widetilde{N}d(u,g^{-1}%
\mathbf{p}_{|\widetilde{k}})]=\psi _{\widetilde{k},\widetilde{N}}^{M}(u)%
\text{.}
\end{equation*}%
Fix $k,N\in \omega $ with $N\geq 1$. Observe that for every $\widetilde{N}%
,m,k\in \omega $%
\begin{equation*}
\inf_{y_{0},\ldots ,y_{\widetilde{k}-1}}[m\tau _{\widetilde{k}}(\bar{y},%
\mathbf{p}_{|\widetilde{k}})+Nd(\bar{y}_{|k},\bar{x})+\psi _{\widetilde{k},%
\widetilde{N}}\left( \bar{y}\right) ]
\end{equation*}%
is an $N$-Lipschitz formula in the $k$ free variables $\bar{x}$. Therefore 
\begin{equation*}
\inf_{\widetilde{N}\geq N}\inf_{\widetilde{k}\geq k}\sup_{m\geq
1}\inf_{y_{0},\ldots ,y_{\widetilde{k}-1}}[m\tau _{\widetilde{k}}(\bar{y},%
\mathbf{p}_{|\widetilde{k}})+Nd(\bar{y}_{|k},\bar{x})+\psi _{\widetilde{k},%
\widetilde{N}}\left( \bar{y}\right) ]
\end{equation*}%
is also an $N$-Lipschitz formula $\varphi \left( \bar{x}\right) $ in the $k$
free variables $\bar{x}$. Moreover it follows from Lemma \ref{Lemma:
negation transform} that Equation \eqref{Equation:inf} holds. Conversely
suppose that for every $k,N\in \omega $ with $N\geq 1$ there exists and $N$%
-Lipschitz formula $\varphi $ such that Equation \eqref{Equation:inf} holds.
Performing the substitution $x\mapsto 1-x$ in each side of Equation %
\eqref{Equation:inf} shows that $1-A\in \mathcal{B}$. By the proof above
applied to $1-A$ we conclude that for every $k,N\in \omega $ with $N\geq 1$
there is an $N$-Lipschitz formula $\psi $ such that 
\begin{equation*}
\inf\nolimits_{g\in \Iso(\mathbb{U})}^{\ast }[(1-A)(gM)+Nd(g^{-1}\mathbf{p}%
_{|k},u)]=\psi ^{M}(u)\text{.}
\end{equation*}%
Performing again the substitution $x\mapsto 1-x$ gives%
\begin{equation*}
\sup\nolimits_{g\in \Iso(\mathbb{U})}^{\ast }[A\left( gM\right) -Nd(g^{-1}%
\mathbf{p}_{|k},u)]=\left( 1-\psi \right) ^{M}\left( u\right) \text{.}
\end{equation*}
Therefore the fomula $1-\psi $ witnesses the fact that $A\in \mathcal{B}$.
\end{proof}

With a similar argument as in the proof of Lemma \ref{Lemma: inf} one can
prove the following lemma.

\begin{lemma}
\label{Lemma: negation in B}Suppose that $A$ is a grey subset of $\Mod(%
\mathcal{L},\mathbb{U})$ and $q\in \left[ 0,1\right] $. Then $A\in \mathcal{B%
}$ if and only if $q-A\in \mathcal{B}$.
\end{lemma}

\subsection{The base case\label{Subsection: base case}}

The proofs of Lemma \ref{Lemma: base case computation 1} and Lemma \ref%
{Lemma: base computation 2} are analogous to the proof of Lemma \ref{Lemma:
negation transform}, and are omitted. The key point is that one can perturb
an element of $\mathbb{U}^{k}$ for which $\tau _{k}(w,\mathbf{p}_{|k})$ is
small to an element in the orbit $\mathcal{O}(\mathbf{p}_{|k})$ of $\mathbf{p%
}_{|k}$ using Lemma \ref{Lemma: perturbation}.

\begin{lemma}
\label{Lemma: base case computation 1}Suppose that $n,k\in \omega $, $\phi $
is a quantifier-free $\mathcal{L}_{\omega \omega }$-formula with $n$ free
variables, and $u\in \mathbb{U}^{k}$. If $n<k$ and $t\in \left[ 0,1\right] $%
, then following statements are equivalent:

\begin{enumerate}
\item $\inf\nolimits_{g\in \Iso(\mathbb{U})}^{\ast }[\phi ^{M}(g^{-1}\mathbf{%
p}_{|n})+Nd(g^{-1}\mathbf{p}_{|k},u)]<t$;

\item There is $w\in \mathcal{O}(\mathbf{p}_{|k})$ such that $\phi
^{M}(w_{|n})+Nd\left( w,u\right) <t$;

\item There is $t_{0}<t$ such that for every $m\geq 1$, $\exists w\in 
\mathbb{U}^{k}$ such that $\phi ^{M}(w_{|n})+Nd\left( w,u\right) +m\tau
_{k}(w,\mathbf{p}_{|k})<t_{0}$.
\end{enumerate}
\end{lemma}

\begin{lemma}
\label{Lemma: base computation 2}Suppose that $n,k\in \omega $, $\phi $ is a
quantifier-free $\mathcal{L}_{\omega \omega }$-formula with $n$ free
variables, and $u\in \mathbb{U}^{k}$. If $k\leq n$ and $t\in \left[ 0,1%
\right] $, then the following statements are equivalent:

\begin{enumerate}
\item $\inf\nolimits_{g\in \Iso(\mathbb{U})}^{\ast }[\phi ^{M}(g^{-1}\mathbf{%
p}_{|n})+Nd(g^{-1}\mathbf{p}_{|k},u)]<t$;

\item There is $w\in \mathcal{O}(\mathbf{p}_{|n})$ such that $\phi
^{M}(w)+Nd\left( w_{|k},u\right) <t$;

\item There is $t_{0}<t$ such that for every $m\geq 1$ there is $w\in 
\mathbb{U}^{n}$ such that $\phi ^{M}(w)+Nd\left( w_{|k},u\right) +m\tau
_{n}(w,\mathbf{p}_{|n})<t_{0}$.
\end{enumerate}
\end{lemma}

\begin{lemma}
\label{Lemma: base case}If $\phi $ is an $\mathcal{L}_{\omega \omega }$
quantifier-free formula with $n$ free variables, then the grey subset $%
M\mapsto \phi ^{M}(g^{-1}\mathbf{p}_{|n})$ of $\Mod(\mathcal{L},\mathbb{U})$
belongs to $\mathcal{B}$.
\end{lemma}

\begin{proof}
By Lemma \ref{Lemma: inf} it is enough to show that for every $N,k\in \omega 
$ with $N\geq 1$ there is an $N$-Lipschitz $\mathcal{L}_{\omega _{1}\omega }$
formula $\psi $ with $k$ free variables such that%
\begin{equation*}
\inf\nolimits_{g\in \Iso(\mathbb{U})}^{\ast }[\phi ^{M}(g^{-1}\mathbf{p}%
_{|n})+Nd(g^{-1}\mathbf{p}_{|k},u)]=\psi ^{M}(u)
\end{equation*}%
for every $M\in \Mod(\mathcal{L},\mathbb{U})$ and $u\in \mathbb{U}^{k}$. Let
us distinguish the cases when $n<k$ and $n\geq k$. If $n<k$ define the $N$%
-Lipschitz formula $\psi \left( \bar{x}\right) $ in the $k$ free variables $%
\bar{x}$ by%
\begin{equation*}
\sup_{m\geq 1}\inf_{y_{0},\ldots ,y_{k-1}}[m\tau _{k}(\bar{y},\mathbf{p}%
_{|k})+Nd\left( \bar{y},\bar{x}\right) +\phi (\bar{y}_{|n})]\text{.}
\end{equation*}%
It follows from Lemma \ref{Lemma: base case computation 1} that for every $%
M\in \Mod(\mathcal{L},\mathbb{U})$ and $u\in \mathbb{U}^{k}$%
\begin{equation*}
\inf\nolimits_{g\in \Iso(\mathbb{U})}^{\ast }[\phi ^{M}(g^{-1}\mathbf{p}%
_{|n})+Nd(g^{-1}\mathbf{p}_{|k},u)]=\psi ^{M}(u)\text{.}
\end{equation*}%
If $n\geq k$ then define $\psi \left( \bar{x}\right) $ to be the $N$%
-Lipschitz formula in the $k$ free variables $\bar{x}$%
\begin{equation*}
\sup_{m\geq 1}\inf_{y_{0},\ldots ,y_{n-1}}[m\tau _{n}(\bar{y},\mathbf{p}%
_{|n})+Nd(\bar{x},\bar{y}_{|k})+\phi \left( \bar{y}\right) ]\text{.}
\end{equation*}%
It follows from Lemma \ref{Lemma: base computation 2} that for every $M\in %
\Mod(\mathcal{L},\mathbb{U})$ and $u\in \mathbb{U}^{k}$%
\begin{equation*}
\inf\nolimits_{g\in \Iso(\mathbb{U})}^{\ast }[\phi ^{M}(g^{-1}\mathbf{p}%
_{|n})+Nd(g^{-1}\mathbf{p}_{|k},u)]=\psi ^{M}(u)\text{,}
\end{equation*}%
which concludes the proof.
\end{proof}

\subsection{Linear combinations\label{Subsection: linear combinations}}

The proof of Lemma \ref{Lemma: transform linear combination} is analogous to
the proofs of Lemmas~\ref{Lemma: negation transform}, \ref{Lemma: base case
computation 1}, and \ref{Lemma: base computation 2}, and it is again omitted.

\begin{lemma}
\label{Lemma: transform linear combination}Suppose that $A,B$ are grey
subsets of $\Mod(\mathcal{L},\mathbb{U})$, $k,N\in \omega $ with $N\geq 1$,
and $\lambda ,\mu \in \left[ 0,1\right] $. For any $t\in \left[ 0,1\right] $%
, $u\in \mathbb{U}^{k}$, and $M\in \Mod(\mathcal{L},\mathbb{U})$, the
following statements are equivalent:

\begin{enumerate}
\item $\inf_{g\in \Iso(\mathbb{U})}^{\ast }[\left( \lambda A+\mu B\right)
(gM)+Nd(g^{-1}\mathbf{p}_{|k},u)]<t$;

\item There are $\widetilde{k}\geq k$ and $\widetilde{N}\geq N$ and $%
\widetilde{u}\in \mathcal{O}(\mathbf{p}_{|\widetilde{k}})$ such that 
\begin{multline*}
Nd\left( \widetilde{u}_{|k},u\right) +\lambda \sup\nolimits_{g\in \Iso(%
\mathbb{U})}^{\ast }[A(gM)-\widetilde{N}d(\widetilde{u},g^{-1}\mathbf{p}_{|%
\widetilde{k}})] \\
+\mu \sup\nolimits_{g\in \Iso(\mathbb{U})}^{\ast }[B(gM)-\widetilde{N}d(%
\widetilde{u},g^{-1}\mathbf{p}_{|\widetilde{k}})]<t\text{;}
\end{multline*}

\item There are $t_{0}<t$, $\widetilde{k}\geq k$, $\widetilde{N}\geq 
\widetilde{N}$ such that for every $m\geq 1$ there is $\widetilde{u}\in 
\mathbb{U}^{\widetilde{k}}$ such that%
\begin{multline*}
Nd\left( \widetilde{u}_{|k},u\right) +\lambda \sup\nolimits_{g\in \Iso(%
\mathbb{U})}^{\ast }[A(gM)-\widetilde{N}d(\widetilde{u},g^{-1}\mathbf{p}_{|%
\widetilde{k}})] \\
+\mu \sup\nolimits_{g\in \Iso(\mathbb{U})}^{\ast }[B(gM)-\widetilde{N}d(%
\widetilde{u},g^{-1}\mathbf{p}_{|\widetilde{k}})]+m\tau _{\widetilde{k}}(%
\widetilde{u},\mathbf{p}_{|\widetilde{k}})<t_{0}\text{.}
\end{multline*}
\end{enumerate}
\end{lemma}

\begin{lemma}
\label{Lemma: linear combinations in B}Suppose that $A,B$ are grey subsets
of $\Mod(\mathcal{L},\mathbb{U})$ that belong to $\mathcal{B}$. If $\lambda
,\mu \in \left[ 0,1\right] $, then $\lambda A+\mu B$ belongs to $\mathcal{B}$%
.
\end{lemma}

\begin{proof}
Since $A,B\in \mathcal{B}$, for every $\widetilde{k},\widetilde{N}\in \omega 
$ with $\widetilde{N}\geq 1$ there are $\widetilde{N}$-Lipschitz formulas $%
\psi _{A,\widetilde{k},\widetilde{N}}$ and $\psi _{B,\widetilde{k},%
\widetilde{N}}$ in $\widetilde{k}$ free variables such that%
\begin{equation*}
\sup\nolimits_{g\in \Iso(\mathbb{U})}^{\ast }A[(gM)-\widetilde{N}d(%
\widetilde{u},g^{-1}\mathbf{p}_{|\widetilde{k}})]=\psi _{A,\widetilde{k},%
\widetilde{N}}^{M}(\widetilde{u})
\end{equation*}%
and similarly for $B$. Fix $k,N\in \omega $ with $N\geq 1$, and define the $%
N $-Lipschitz formula $\varphi \left( \bar{x}\right) $ in the $k$ free
variables $\bar{x}$ 
\begin{equation*}
\inf_{\widetilde{N}\geq N}\inf_{\widetilde{k}\geq k}\sup_{m\geq
1}\inf_{y_{0},\ldots ,y_{\widetilde{k}-1}}[m\tau _{\widetilde{k}}(\bar{y},%
\mathbf{p}_{|\widetilde{k}})+Nd(\bar{y}_{|k},\bar{x})+\lambda \psi _{A,%
\widetilde{k},\widetilde{N}}\left( \bar{y}\right) +\mu \psi _{B,\widetilde{k}%
,\widetilde{N}}\left( \bar{y}\right) ]\text{.}
\end{equation*}%
By Lemma \ref{Lemma: transform linear combination} for $M\in \Mod(\mathcal{L}%
,\mathbb{U})$ and $u\in \mathbb{U}^{k}$%
\begin{equation*}
\inf\nolimits_{g\in \Iso(\mathbb{U})}^{\ast }[\left( \lambda A+\mu B\right)
(gM)+Nd(g^{-1}\mathbf{p}_{|k},u)]=\varphi ^{M}(u)\text{.}
\end{equation*}%
In view of Lemma \ref{Lemma: inf} this concludes the proof that $\lambda
A+\mu B\in \mathcal{B}$.
\end{proof}

\subsection{Infima and suprema\label{Subsection: infima and suprema}}

\begin{lemma}
\label{Lemma: sup and inf in B}If $\left( A_{n}\right) _{n\in \omega }$ is a
sequence of grey subsets of $\Mod(\mathcal{L},\mathbb{U})$ that belong to $%
\mathcal{B}$, then $\inf_{n}A_{n}$ and $\sup_{n}A_{n}$ belong to $\mathcal{B}
$.
\end{lemma}

\begin{proof}
By Lemma \ref{Lemma: negation in B} it is enough to show that $%
\sup_{n}A_{n}\in \mathcal{B}$. Fix $k,N\in \omega $ with $N\geq 1$. For
every $n\in \omega $, since $A_{n}\in \mathcal{B}$ there is an $N$-Lipschitz
formula $\varphi _{n}$ such that for every $M\in \Mod(\mathcal{L},\mathbb{U}%
) $ and $u\in \mathbb{U}^{k}$%
\begin{equation*}
\sup\nolimits_{g\in \Iso(\mathbb{U})}^{\ast }[A_{n}(gM)-Nd(u,g^{-1}\mathbf{p}%
_{|k})]=\varphi _{n}^{M}(u)\text{.}
\end{equation*}%
It follows from Proposition \ref{Proposition: grey quantifiers}(3) that%
\begin{align*}
\sup\nolimits_{g\in \Iso(\mathbb{U})}^{\ast }[\sup_{n}A_{n}(gM)-Nd(u,g^{-1}%
\mathbf{p}_{|k})]& =\sup_{n}\sup\nolimits_{g\in \Iso(\mathbb{U})}^{\ast
}[A_{n}(gM)-Nd(u,g^{-1}\mathbf{p}_{|k})] \\
& =\left( \sup\nolimits_{n}\varphi _{n}\right) ^{M}(u)\text{.}
\end{align*}%
Since $\sup_{n}\varphi _{n}$ is an $N$-Lipschitz formula, this shows that $%
\sup_{n}A_{n}\in \mathcal{B}$.
\end{proof}

\bibliographystyle{amsplain}
\bibliography{bib-lop-esc}

\end{document}